\documentclass[11pt]{article}
\usepackage{enumerate}
\usepackage{verbatim}
\usepackage[english]{babel}
\usepackage[utf8]{inputenc}
\usepackage{listings}
\usepackage{authblk}
\usepackage{geometry}
\usepackage{psfrag}
\usepackage{epic}
\usepackage{lipsum}
\usepackage{eepic}
\usepackage{color}
\usepackage{graphicx}
\usepackage{indentfirst}
\usepackage[caption=false]{subfig}
\usepackage{float}
\usepackage{amsmath}
\usepackage{amssymb}
\usepackage{amsthm}
\usepackage{hyperref}

\usepackage[normalem]{ulem}
\usepackage{indentfirst}

\newtheorem{theo}{Theorem}[section]

\newtheorem{coro}[theo]{Corollary}

\newtheorem{lemma}[theo]{Lemma}
\newtheorem{conjec}[theo]{Conjecture}

\newtheorem{question}[theo]{Question}

\renewcommand{\qed}{$\blacksquare$}

\newcommand{\sC}{\mathcal{C}}

\newcommand{\sP}{\mathcal{P}}

\begin{document}

\title{A family of counterexamples for a conjecture of Berge on $\alpha$-diperfect digraphs}

\author[1]{Caroline Aparecida de Paula Silva \thanks{This author was financed by Coordenação de Aperfeiçoamento de Pessoal de Nível Superior - Brasil (CAPES) - Finance Code 001 and FAPESP Proc. 2020/06116-4. ORCID: 0000-0003-4661-2822.}}
\author[2]{Cândida Nunes da Silva \thanks{ORCID: 0000-0002-4649-0274.}}
\author[1]{Orlando Lee \thanks{This author was financed by CNPq
Proc. 303766/2018-2, CNPq Proc 425340/2016-3 and FAPESP Proc. 2015/11937-9.
ORCID: 0000-0003-4462-3325.}}
\affil[1]{\footnotesize{Institute of Computing, University of Campinas, Campinas, São Paulo, Brazil} \newline \{caroline.silva, lee\}@ic.unicamp.br}
\affil[2]{\footnotesize{Department of Computing, Federal University of São Carlos, Sorocaba, São Paulo, Brazil} candida@ufscar.br}

\date{}

\maketitle 
   
\begin{abstract}
Let $D$ be a digraph. A stable set $S$ of $D$ and a path partition $\sP$ of $D$ are \emph{orthogonal} if every path $P \in \sP$ contains exactly one vertex of $S$. In 1982, Berge defined the class of $\alpha$-diperfect digraphs. A digraph $D$ is \emph{$\alpha$-diperfect} if for every maximum stable set $S$ of $D$ there is a path partition $\sP$ of $D$ orthogonal to $S$ and this property holds for every induced subdigraph of $D$. 
An \emph{anti-directed odd cycle} is an orientation of an odd cycle $(x_0,\ldots,x_{2k},x_0)$ with $k\geq2$ in which each vertex $x_0,x_1,x_2,x_3,x_5,x_7,\ldots,x_{2k-1}$ is either a source or a sink. Berge conjectured that a digraph $D$ is $\alpha$-diperfect if and only if $D$ does not contain an anti-directed odd cycle as an induced subdigraph. In this paper, we show that this conjecture is false by exhibiting an infinite family of orientations of complements of odd cycles with at least seven vertices that are not $\alpha$-diperfect.
\end{abstract}

\section{Introduction}

Let $G=(V(G), E(G))$ be a graph. We use the concepts of path and cycle as defined in~\cite{bomu08}. We may think of a path or cycle as a subgraph of $G$. The \emph{length} of a path (respectively, cycle) is its number of edges. The \emph{order of a path} $P$, denoted by $|P|$, is defined as its number of vertices, that is, $|P| = |V(P)|$. Similarly, the \emph{order of a cycle} is its number of vertices. Let $C_k$ denote the graph isomorphic to a cycle of length $k\geq 3$ and let $\overline{G}$ denote the \emph{complement of $G$.}
We also use the concepts of stable set and clique as defined in~\cite{bomu08}. The cardinality of a maximum stable set (respectively, maximum clique) is denoted by $\alpha(G)$ (respectively, $\omega(G)$).
A \emph{(proper) coloring} $\sC = \{C_{1}, C_{2}, ..., C_{m}\}$ 
of a graph $G$ is a partition of $V(G)$ into stable sets. The cardinality of a minimum coloring is denoted by $\chi{(G)}$.

Let $D=(V(D), A(D))$ be a digraph. 
For every concept for graphs, we may have an analogue for digraphs. The \emph{underlying graph} of $D$, denoted by $U(D)$, is the simple graph with vertex set $V(D)$ such that $u$ and $v$ are adjacent in $U(D)$ if and only if $(u, v) \in A(D)$ or $(v, u) \in A(D)$. We borrow terminology from undirected graphs when dealing with a digraph $D$ by considering its underlying graph $U(D)$. For example, we say that a \emph{stable set} of a digraph $D$ is a stable set of its underlying graph $U(D)$.
Conversely, we may obtain a directed graph $D$ from a graph $G$ by replacing each edge $uv$ of $G$ by an arc $(u,v)$, or an arc $(v,u)$, or both; such directed graph $D$ is called a \emph{super-orientation} of $G$. 
A super-orientation which does not contain a \emph{digon} (a directed cycle of length two) is an \emph{orientation}. A digraph $D$ is \emph{symmetric} if $D$ is a super-orientation of a graph $G$ in which every edge $uv$ of $G$ is replaced by both arcs $(u,v)$ and $(v,u)$.

If $(u,v)$ is an arc of $D$, then we say that $u$ \emph{dominates} $v$ and $v$ is \emph{dominated} by $u$. If $v$ is not dominated by any of its neighbors, then we say that $v$ is a \emph{source}. Similarly, if $v$ does not dominate any of its neighbors in $D$, then we say that $v$ is a \emph{sink}.
A \emph{directed path} or \emph{directed cycle} is an orientation of a path or cycle, respectively, in which each vertex dominates its successor in the sequence. 
Henceforth, when we say path of a digraph, we mean directed path (note that we do not use this con\-ven\-tion for cycles).  We denote by $\lambda(G)$ ($\lambda(D)$) the cardinality of a maximum path in a graph (digraph). When we say a cycle of a digraph, we mean either a super-orientation of an undirected cycle with length at least three or a digon. 

Let $X$ and $Y$ be two disjoint subsets of $V(D)$. We use the notation $X \mapsto Y$ to denote that every vertex of $X$ dominates every vertex of $Y$ in $D$ and no vertex of $Y$ dominates a vertex of $X$ in $D$. If $X=\{u\}$ (respectively, $Y=\{v\}$), we may denote $u \mapsto Y$ (respectively, $X \mapsto v$).
A \emph{path partition} of $D$ is a collection of vertex-disjoint paths of $D$ that cover $V(D)$. Let $\pi(D)$ denote the cardinality of a smallest path partition of $D$. 
In 1960, Gallai and Milgram~\cite{GallaiMilgram1960} showed that, for every digraph, the size of a minimum path partition $\pi(D)$ is less than or equal to the size of a maximum stable set $\alpha(D)$. Actually, Gallai and Milgram showed a stronger statement that implies Gallai-Milgram's Theorem. It uses the concept of \emph{orthogonality}, defined next. Let $\sP$ be
a path partition and let $S$ be a stable set of $D$. We say that $\sP$ and $S$ are
\emph{orthogonal} if $|S \cap P| = 1$ for every $P \in \sP$; we also say that $S$ is orthogonal
to $\sP$ or vice versa. 


\begin{theo}[Gallai and Milgram~\cite{GallaiMilgram1960}]
  \label{th:gallai-milgram}
  Let $D$ be a digraph. For every minimum path partition $\sP$ of $D$, there is a stable set $S$ such that $\sP$ and $S$ are orthogonal. In particular, $\pi(D) \leq \alpha(D)$.
\end{theo}

A straightforward application of Gallai-Milgram's Theorem is the following corollary.

\begin{coro}
\label{coro:join-paths}
Let $D$ be a digraph. Let $P$ be a path of $D$ and let $v \in V(D) - V(P)$. If $v$ is adjacent to every vertex of $P$, then $D$ has a path $P'$ such that $V(P') = V(P) \cup \{v\}$.
\hfill\qed
\end{coro}


A graph $G$ is \emph{perfect} if $\chi(H) = \omega(H)$ for every induced subgraph $H$ of $G$. It is easy to show that if $G$ is perfect, then $G$ cannot contain either an odd cycle of order at least five or its complement as an induced subgraph. 
Berge~\cite{berge1961farbung} conjectured that the converse was true as well.
In 2006, Chudnovsky, Robertson, Seymour and Thomas~\cite{chudnovsky2006strong} proved this long standing open conjecture and it became known as the Strong Perfect Graph Theorem:
\begin{theo}[Chudnovsky \textit{et al}~\cite{chudnovsky2006strong}]
\label{theo:strong-perfect}
A graph $G$ is perfect, if and only if, $G$ does not contain an odd cycle with five or more vertices or its complement as an induced subgraph.
\end{theo}

Motivated by Gallai-Milgram's Theorem and looking for stronger properties in the relationship between stable sets and paths in digraphs, Berge~\cite{Berge1982b} introduced in 1982
a new class of digraphs which he called $\alpha$-diperfect digraphs. A digraph $D$ is
$\alpha$\emph{-diperfect} if every induced subdigraph $H$ of $D$ has the
following property: for every  maximum stable set $S$ of $H$, there exists a
path partition $\sP$ of $H$ such that $\sP$ and $S$ are orthogonal. Berge~\cite{Berge1982b} proved that every symmetric digraph as well as every digraph whose underlying graph is perfect is $\alpha$-diperfect. However, he also showed that 
there are super-orientations of odd cycles that are not $\alpha$-diperfect.
We say that a super-orientation $D$ of an odd cycle is an \emph{anti-directed odd cycle} if $U(D) =(y_0,\ldots,y_{2k},y_0)$ with $k\geq2$ and each of $y_0,y_1,y_2,y_3,y_5,y_7,\ldots,y_{2k-1}$ is either a source or a sink in $D$. Figure~\ref{fig:anti-directed-odd-cycles} shows examples of anti directed odd cycles.

\begin{figure}[htb]
    \centering
    \subfloat[The set $S=\{y_0, y_3\}$ is a maximum stable set but there is no path partition orthogonal to $S$.]{\includegraphics{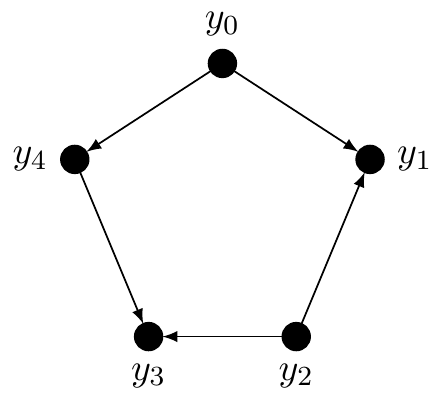}}
    \qquad
    \qquad
    \subfloat[The set $S=\{y_0, y_3, y_5\}$ is a maximum stable set but there is no path partition orthogonal to $S$.]{\includegraphics{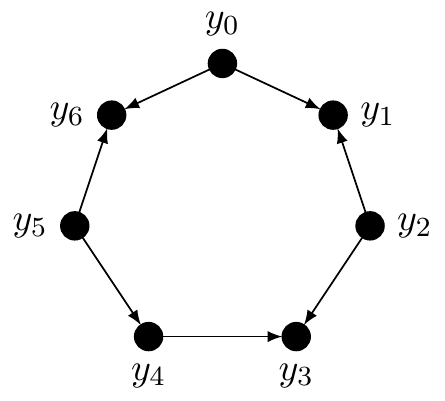}}
    \caption{Examples of anti-directed odd cycles.}
    \label{fig:anti-directed-odd-cycles}
\end{figure}

Berge~\cite{Berge1982b} proved the following characterization of super-orientations of odd cycles with at least five vertices that are $\alpha$-diperfect.

\begin{theo}[Berge~\cite{Berge1982b}]\label{th:anti-directed}
Let $D$ be a super-orientation of a $C_{2k+1}$, with $k\geq2$. Then, $D$ is $\alpha$-diperfect if and only if $D$ is not an anti-directed odd cycle.
\end{theo}

Analogously to Theorem~\ref{theo:strong-perfect}, Berge was interested in obtaining a characterization of the class of $\alpha$-diperfect digraphs in terms of forbidden subdigraphs.
In fact, he proposed the following conjecture.

\begin{conjec}[Berge~\cite{Berge1982b}]
\label{conj:berge-alpha-dip}
A digraph $D$ is $\alpha$-diperfect if and only if $D$ does not contain an anti-directed odd cycle as an induced subdigraph.
\end{conjec}

Motivated by Berge's Conjecture, Sambinelli, Silva and Lee~\cite{tesemaycon2018} proposed in 2018 a similar conjecture. Before we state it, we need some definitons. A digraph $D$ is \emph{Begin-End-diperfect} or simply \emph{BE-diperfect} if every induced subdigraph $H$ of $D$ satisfies the following properties: (i) for every maximum stable set $S$ of $D$ there is a path partition $\sP$ of $D$ orthogonal to $S$ and (ii) every path $P \in \sP$ starts or ends at a vertex of $S$.
We say that a super-orientation $D$ of an odd cycle is a \emph{blocking odd cycle} if $U(D) =(y_0,\ldots,y_{2k},y_0)$ with $k\geq1$ and each of $y_0$ and $y_1$ is either a source or a sink in $D$. Figure~\ref{fig:blocking-odd-cycles} shows examples of blocking odd cycles.

\begin{figure}[htb]
    \centering
    \subfloat[The set $\{y_2\}$ is a maximum stable set but there is no hamiltonian path which starts or ends at $y_2$.]{\includegraphics{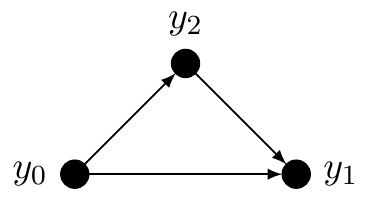}}
    \qquad
    \qquad
    \subfloat[The set $S=\{y_2, y_4\}$ is a maximum stable set but there is no path partition $\sP$ orthogonal to $S$ in which every path $P \in \sP$ starts or ends at a vertex of $S$. ]{\includegraphics{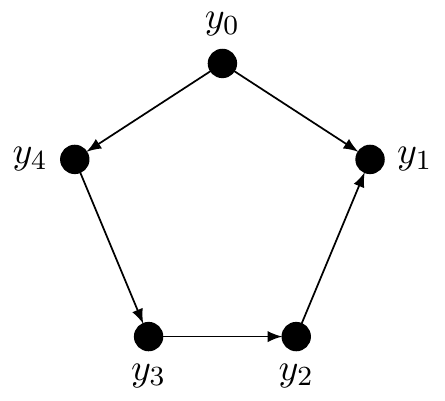}}
    \caption{Examples of blocking odd cycles.}
    \label{fig:blocking-odd-cycles}
\end{figure} 

\begin{conjec}[Sambinelli, Silva and Lee~\cite{tesemaycon2018}]
\label{conj:be-dip}
A digraph $D$ is $BE$-diperfect if and only if $D$ does not contain a blocking odd cycle as an induced subdigraph.
\end{conjec}

Recently, Conjectures~\ref{conj:berge-alpha-dip} and~\ref{conj:be-dip} were verified for some specific classes of digraphs (see~\cite{sambinelli2022alpha}, \cite{freitas2022some} and \cite{freitas2022-3anti}).

In this paper, we show that Berge's Conjecture (Conjecture ~\ref{conj:berge-alpha-dip}) is false for arbitrary digraphs. We present and characterize super-orientations of complements of odd cycles with at least five vertices that are not $\alpha$-diperfect. One can easily check that a complement of an odd cycle with at least five vertices cannot contain an induced odd cycle with at least five vertices. Thus every super-orientation of the complement of an odd cycle is free from anti-directed odd cycles.
On the other hand, all these counterexamples to Conjecture~\ref{conj:berge-alpha-dip} contain blocking odd cycles as induced subdigraphs. So these digraphs are not counterexamples to Conjecture~\ref{conj:be-dip}.
In fact, it can be shown that a super-orientation $D$ of the complement of an odd cycle with at least five vertices is BE-diperfect if and only if $D$ does not contain an blocking odd cycle as an induced subdigraph~\cite{be-comp-ciclo-impar-2022}.
 
\section{$\alpha$-Diperfect super-orientations of $\overline{C_{2k+1}}$}

We start this section by presenting a necessary condition for a digraph to be $\alpha$-diperfect.

\begin{lemma}
\label{lem:lambda-alpha}
Let $D$ be a digraph and let $S$ be a maximum stable set of $D$ of size at least two. If $D$ is $\alpha$-diperfect, then 
$\lambda(D - u) \geq \Bigl\lceil\frac{|V(D)|}{\alpha(D)}\Bigr\rceil$ or $\lambda(D - v) \geq \Bigl\lceil\frac{|V(D)|}{\alpha(D)}\Bigl\rceil$, 
for every pair of distinct vertices $u, v \in S$.
\end{lemma}
\begin{proof}
Let $\sP$ be a path partition orthogonal 
to $S$. Clearly, at least one path of $\sP$ must have size at least $\Bigl\lceil\frac{|V(D)|}{\alpha(D)}\Bigr\rceil$. Since $u$ and $v$ belong to distinct paths of $\sP$, the result follows.
\end{proof}

Let $G$ be a graph isomorphic to $\overline{C_{2k+1}}$, with $k\geq2$ and let $D$ be a super-orientation of $D$.
Henceforth, we may assume that the vertices of $G$ (and of $D$) are labeled as $x_0,\ldots,x_{2k}$ so that the cycle $\overline{G} = \overline{U(D)}$ is $(x_0,\ldots,x_{2k}, x_0)$. So, the non-neighbors of $x_i$ are $x_{i-1}$ and $x_{i+1}$, where the indexes are taken modulo $2k+1$, as depicted in Figure~\ref{fig:0-label-barc9}. Moreover, note that $\alpha(G)=\alpha(D)=2$ and each pair $\{x_i, x_{i+1}\}$ is a maximum stable set of $G$ (and hence, of $D$).

\begin{figure}[htb]
    \centering
    \subfloat{\includegraphics{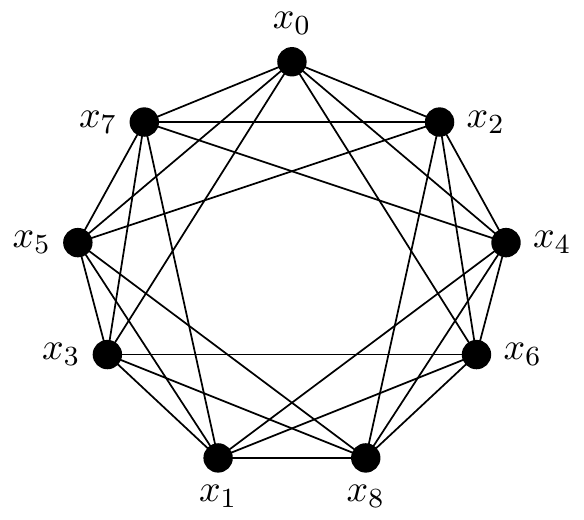}}
    \caption{Labeling for $\overline{C_9}$.}
    \label{fig:0-label-barc9}
\end{figure}

We say that $D$ is a \emph{$\vec{D}_{2k+1}$} if  there is no arc $(x_j, x_i)$, for $0 \leq i < j \leq  2k$. 
Figure~\ref{fig:x0x2k-path} shows the digraphs $\vec{D}_{5}$ and $\vec{D}_{7}$.


\begin{figure}[htb]
    \centering
    \subfloat[\label{fig:anti-c5}Digraph $\vec{D}_{5}$.]{\includegraphics{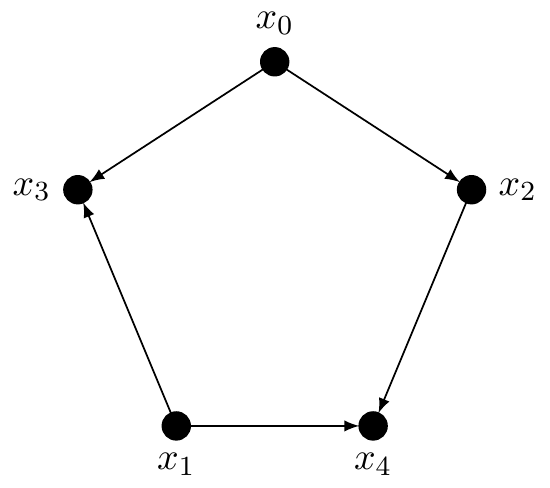}}
    \qquad
    \qquad
    \subfloat[Digraph $\vec{D}_{7}$. ]{\includegraphics{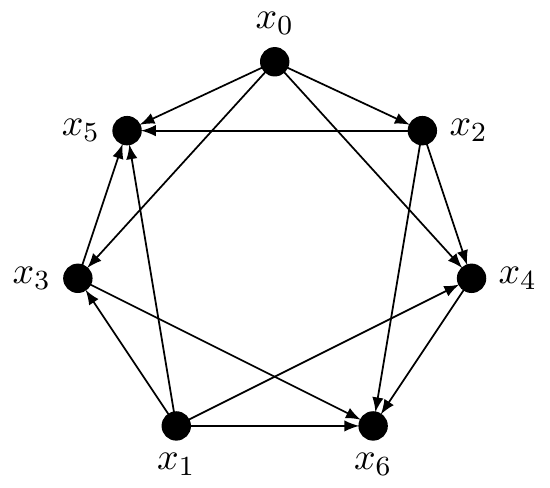}}
    \caption{Examples of $\vec{D}_{2k+1}$.}
    \label{fig:x0x2k-path}
\end{figure}

Our goal is to show that a super-orientation $D$ of the complement of an odd cycle with at least five vertices has a path partition orthogonal to $\{x_0, x_{2k}\}$ if and only if $D$ is not isomorphic to $\vec{D}_{2k+1}$. (Theorem~\ref{th:carac-alpha-dip-barC2k1}). Note that this implies that there exists an infinite family of counterexamples of Berge's conjecture. One may verify that the super-orientation of a $\overline{C_5}$ depicted in Figure~\ref{fig:x0x2k-path}(a) does not have a path partition orthogonal to $\{x_0, x_4\}$. Note that this digraph is also an anti-directed odd cycle. Similarly, the super-orientation of a $\overline{C_7}$ depicted in Figure~\ref{fig:x0x2k-path}(b) does not have a path partition orthogonal to $\{x_0, x_6\}$. 

Let $G$ be a graph isomorphic to $\overline{C_{2k+1}}$ for $k \geq 3$. Let $G'$ be the graph obtained from $G$ by deleting a pair of non-adjacent vertices $x_i$ and $x_{i+1}$ and the edge $x_{i-1}x_{i+2}$. Note that $G'$ is isomorphic to $\overline{C_{2k-1}}$ because $G'$ is the complement of the cycle $(x_0,\ldots,x_{i-1},\allowbreak x_{i+2},\ldots,x_{2k},x_0)$. This observation allows us to obtain a path partition orthogonal to $\{x_0,x_{2k}\}$ of a super-orientation $D$ of $G$ when a certain subdigraph of $D$ admits a specific path partition. 


\begin{lemma}\label{lem:alpha-subdigraph-x2k}
Let $D$ be a super-orientation of a $\overline{C_{2k+1}}$ with $k\geq3$. Let $D'$ be the super-orientation of a $\overline{C_{2k-1}}$ obtained from $D$ by deleting the vertices $x_{2k-1}$ and $x_{2k}$ and the arc between $x_0$ and $x_{2k-2}$. If $D'$ has a path partition orthogonal to $\{x_0,x_{2k-2}\}$, then $D$ has a path partition orthogonal to $\{x_0,x_{2k}\}$.
\end{lemma}
\begin{proof}
Suppose that $D'$ has a path partition $\sP = \{P_1,P_2\}$  orthogonal to $\{x_0, x_{2k-2}\}$. Without loss of generality, we may assume that $x_0 \in V(P_1)$ and $x_{2k-2} \in V(P_2)$. 
Since the only non-neighbors of $x_{2k-1}$ in $D$ are $x_{2k-2}$ and $x_{2k}$, it follows that $x_{2k-1}$ is adjacent, in $D$, to every vertex of $P_1$. So, by Corollary~\ref{coro:join-paths}, there is a path $P_1'$ in $D$ such that $V(P_1') = V(P_1) \cup \{x_{2k-1}\}$. 
Similarly, the only non-neighbors of $x_{2k}$ are $x_0$ and $x_{2k-1}$. So, it follows that $x_{2k}$ is adjacent, in $D$, to every vertex of $P_2$. By Corollary~\ref{coro:join-paths}, there is a path $P_2'$ in $D$ such that $V(P_2') = V(P_2) \cup \{x_{2k}\}$.
Thus, $\{P_1', P_2'\}$ is a path partition of $D$ orthogonal to $\{x_0, x_{2k}\}$.
\end{proof}

By adjusting notation, we immediately have the following corollary.

\begin{coro}\label{coro:alpha-subdigraph-x0}
Let $D$ be a super-orientation of a $\overline{C_{2k+1}}$ for $k\geq3$. Let $D'$ be the super-orientation of a $\overline{C_{2k-1}}$ obtained from $D$ by deleting the vertices $x_0$ and $x_1$ and the arc between $x_2$ and $x_{2k}$. If $D'$ has a path partition orthogonal to $\{x_2,x_{2k}\}$, then $D$ has a path partition orthogonal to $\{x_0,x_{2k}\}$.
\hfill \qed
\end{coro}

\begin{theo}\label{th:carac-alpha-dip-barC2k1}
Let $D$ be a super-orientation of a $\overline{C_{2k+1}}$ for $k\geq2$. Then, $D$ has a path partition orthogonal to $\{x_0, x_{2k}\}$ if and only if $D$ is not isomorphic to  $\vec{D}_{2k+1}$.
\end{theo}
\begin{proof}
(Necessity) 
Assume that $D$ is isomorphic to $\vec{D}_{2k+1}$. 
We show that
$\lambda(D - x_0) < k+1$ and  $\lambda(D - x_{2k}) < k+1$, and the result follows immediately by Lemma~\ref{lem:lambda-alpha}.
Note that, by the Principle of Directional Duality, it suffices to prove that $\lambda(D - x_0) < k+1$.
Let $D' = D - x_0$. 
Let $P=(v_1,\ldots,v_\ell)$ be a longest path of $D'$.  For each $t \in \{1, 2, \ldots, \ell\}$, let $r(t) = s$ if $v_t=x_s$.
By definition, for $j > i$, there is no arc $(x_j, x_i)$. Hence,
the sequence $r(1), r(2), \ldots , r(\ell)$ is strictly increasing. Towards a contradiction, suppose that $\ell > k$. So there must exist $j \in \{1,\ldots,\ell-1\}$ such that $v_j=x_i$ and $v_{j+1}=x_{i+1}$. However, this is a contradiction since $x_i$ is non-adjacent to $x_{i-1}$ and $x_{i+1}$.
Thus, the size of $P$ is at most $k$.

(Sufficiency) Suppose that $D$ does not have a path partition orthogonal to $\{x_0,x_{2k}\}$. We show that $D$ is isomorphic to $\vec{D}_{2k+1}$.
The proof follows by induction on $k$. If $k=2$, then $D$ is a super-orientation of a $\overline{C_5} = C_5$. By Theorem~\ref{th:anti-directed}, $D$ is an anti-directed odd cycle. It is easy to verify that $D$ is isomorphic to $\vec{D}_{5}$ (see Figure~\ref{fig:anti-c5}). So suppose that $k \geq 3$.
Let $D_1$ be the super-orientation of a $\overline{C_{2k-1}}$ obtained from $D$ by deleting the vertices $x_{2k-1}$ and $x_{2k}$ and the arc between $x_0$ and $x_{2k-2}$.
By Lemma~\ref{lem:alpha-subdigraph-x2k}, we may assume that $D_1$ has no path 
partition orthogonal to $\{x_0,x_{2k-2}\}$. Since the vertices of $D_1$ are labeled as $x_0,\ldots,x_{2k-2}$, we may apply the induction hypothesis and assume that $D_1$ is isomorphic to $\vec{D}_{2k-1}$. By the Principle of Directional Duality, we may assume that the following property holds:

\begin{enumerate}[(a)]
    \item for $0 \leq i < j \leq 2k-2$, there is no arc $(x_j, x_i)$ in $D_1$ (and hence, in $D$) (see Figure~\ref{fig:d1}).
\end{enumerate}

\begin{figure}[htb]
    \centering
    \subfloat[\label{fig:d1}Digraph $D_1$ is isomorphic to  $\vec{D}_{2k-1}$.]{\includegraphics{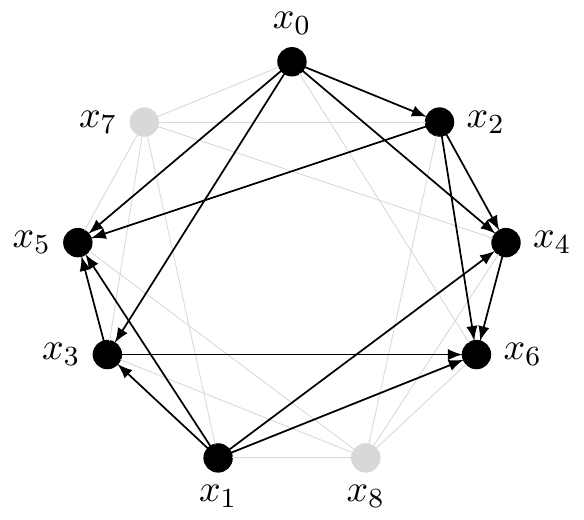}}
    \subfloat[\label{fig:d2}Digraph $D_2$ is isomorphic to  $\vec{D}_{2k-1}$.]{\includegraphics{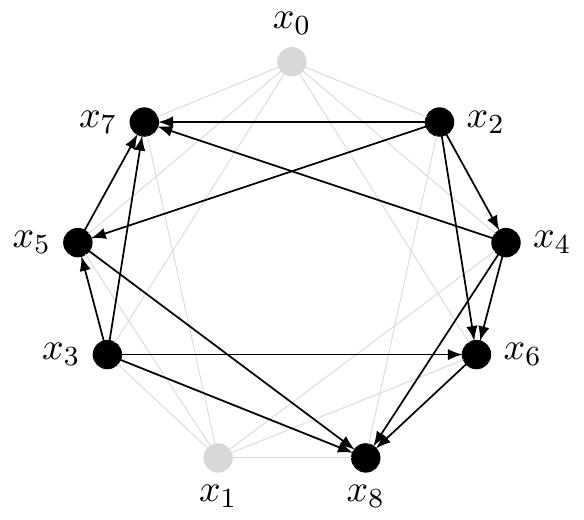}}\\
    \subfloat[\label{fig:d12}$D_1 \cup D_2$.]{\includegraphics{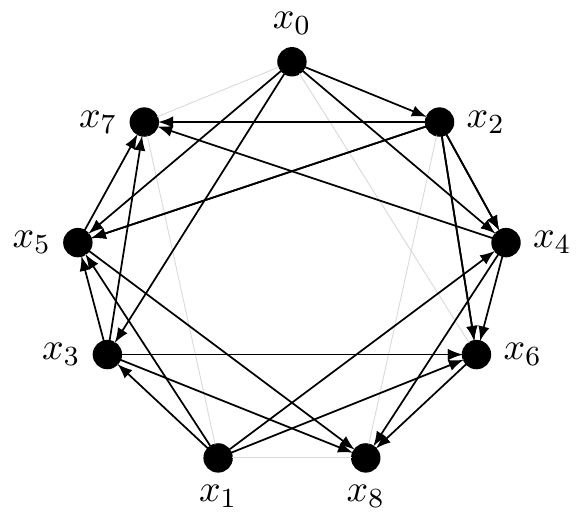}}
      \subfloat[\label{fig:d}Digraph $D$ is isomorphic to  $\vec{D}_{2k+1}$.]{\includegraphics{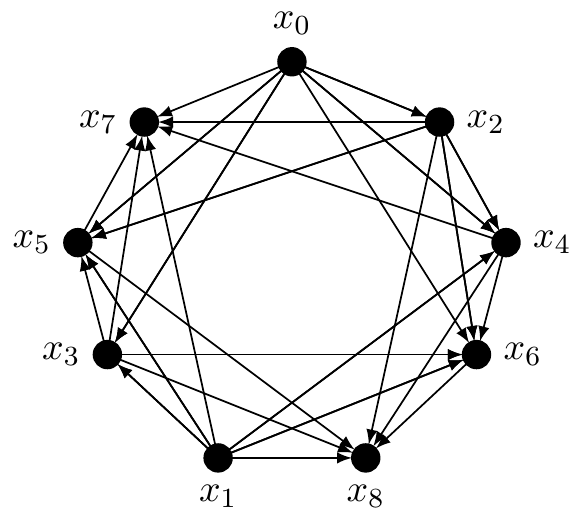}}
    \caption{Auxiliary illustration for the proof of Theorem~\ref{th:carac-alpha-dip-barC2k1}} Examples of $D_1$ and $D_2$ for $\overline{C_9}$ ($k=4$).
\end{figure}

Similarly, let $D_2$ be the super-orientation of a $\overline{C_{2k-1}}$ obtained from $D$ by deleting the vertices $x_0$ and $x_1$ and the arc between $x_2$ and $x_{2k}$. By Corollary~\ref{coro:alpha-subdigraph-x0}, we may assume that $D_2$ has no path partition orthogonal to $\{x_2, x_{2k}\}$. 
Since the vertices of $D_2$ are labeled as $x_2,\ldots,x_{2k}$, we may adjust notation and apply the induction hypothesis. So $D_2$ is also isomorphic to $\vec{D}_{2k-1}$. Moreover, since (a) holds, then $D_2$ must satisfy the following property:
\begin{enumerate}[(b)]
    \item for $2 \leq i < j \leq 2k$, there is no arc $(x_j, x_i)$ in $D_2$ (and hence, in $D$) (see Figure~\ref{fig:d2}).
\end{enumerate}

Let $D' = D_1 \cup D_2$. Note that $V(D) = V(D')$ and $E(U(D)) - E(U(D')) = \{x_0x_{2k-1},\allowbreak x_0x_{2k-2},\allowbreak x_1x_{2k-1},\allowbreak x_1x_{2k-2},\allowbreak x_2x_{2k}\}$ (see Figure~\ref{fig:d12}). 
Since (a) and (b) hold, it suffices to prove that $x_0 \mapsto \{x_{2k-1},x_{2k-2}\}$, $\{x_1,x_2\} \mapsto x_{2k}$ and $x_1 \mapsto x_{2k-1}$.
Towards a contradiction, suppose that $x_{2k-1}$ dominates $x_0$. Then, $\{(x_1,x_3\ldots,x_{2k-1},x_0), (x_2,x_4,\ldots,x_{2k})\}$ is a path partition of $D$ orthogonal to $\{x_0,x_{2k}\}$, a contradiction. Similarly, suppose that $x_{2k}$ dominates $x_1$. Then, $\{(x_{2k},x_1,x_3,\ldots,x_{2k-1}),\allowbreak (x_0,x_2,\ldots,x_{2k-2})\}$ is a path partition of $D$ orthogonal to $\{x_0,x_{2k}\}$, a contradiction.
Now suppose that $x_{2k-2}$ dominates $x_0$. Then $\{(x_1,x_3,\ldots,x_{2k-3},x_{2k}),\allowbreak (x_2,x_4,\ldots,\allowbreak x_{2k-2},\allowbreak x_0, x_{2k-1})\}$ is a path partition of $D$ orthogonal to $\{x_0,x_{2k}\}$, a contradiction. Similarly, suppose that $x_{2k}$ dominates $x_2$. Then, $\{(x_1,x_{2k},x_2,x_4,\ldots,x_{2k-2}), (x_0,x_3,x_5,\ldots,x_{2k-1})\}$ is a path partition of $D$ orthogonal to $\{x_0,x_{2k}\}$, a contradiction.
Finally, suppose that $x_{2k-1}$ dominates $x_1$. Then, $\{(x_3,x_5,\ldots,\allowbreak x_{2k-1}, x_1,  x_{2k}),(x_0,x_2,\ldots,x_{2k-2})\}$ is a path partition of $D$ orthogonal to $\{x_0,x_{2k}\}$, a contradiction.
Thus, $D$ is isomorphic to $\vec{D}_{2k+1}$ (see Figure~\ref{fig:d}).
\end{proof}

One may ask the following natural question.

\begin{question}
Let $D$ be a digraph which does not contain an induced anti-directed odd cycle or an induced $\vec{D}_{2k+1}$, with $k\geq 2$. 
Is it true that $D$ is $\alpha$-diperfect?
\end{question}



\bibliographystyle{abbrv}

\begin{thebibliography}{10}

\bibitem{berge1961farbung}
C.~Berge.
\newblock Farbung von {G}raphen, deren samtliche bzw. deren ungerade {K}reise
  starr sind.
\newblock {\em Wissenschaftliche Zeitschrift}, 1961.

\bibitem{Berge1982b}
C.~Berge.
\newblock Diperfect graphs.
\newblock {\em Combinatorica}, 2(3):213--222, 1982.

\bibitem{bomu08}
J.~A. Bondy and U.~S.~R. Murty.
\newblock {\em Graph Theory}.
\newblock Springer, 2008.

\bibitem{chudnovsky2006strong}
M.~Chudnovsky, N.~Robertson, P.~Seymour, and R.~Thomas.
\newblock The {S}trong {P}erfect {G}raph {T}heorem.
\newblock {\em Annals of Mathematics}, 164:51--229, 2006.

\bibitem{freitas2022-3anti}
L.~I.~B. Freitas and O.~Lee.
\newblock $3$-anti-circulant digraphs are $\alpha$-diperfect and
  {BE}-diperfect.
\newblock {\em Open Journal of Discrete Mathematics}, 2022.

\bibitem{be-comp-ciclo-impar-2022}
L.~I.~B. Freitas and O.~Lee.
\newblock Private communication.
\newblock 2022.

\bibitem{freitas2022some}
L.~I.~B. Freitas and O.~Lee.
\newblock Some results on {B}erge’s conjecture and {B}egin--{E}nd
  {C}onjecture.
\newblock {\em Graphs and Combinatorics}, 38(4):1--23, 2022.

\bibitem{GallaiMilgram1960}
T.~Gallai and A.~N. Milgram.
\newblock Verallgemeinerung eines graphentheoretischen {S}atzes von
  {R}{\'e}dei.
\newblock {\em Acta Sci Math}, 21:181--186, 1960.

\bibitem{tesemaycon2018}
M.~Sambinelli.
\newblock {\em Partition problems in graphs and digraphs}.
\newblock PhD thesis, State University of Campinas - UNICAMP, 2018.

\bibitem{sambinelli2022alpha}
M.~Sambinelli, C.~N. da~Silva, and O.~Lee.
\newblock $\alpha$-{D}iperfect digraphs.
\newblock {\em Discrete Mathematics}, 345(5):112759, 2022.

\end{thebibliography}

\end{document}